\documentclass[12pt]{amsart} 

\usepackage[margin=1in]{geometry}
\usepackage{graphicx} % Required for inserting images
\usepackage{amsmath}
\usepackage{amsfonts, amsthm}
\usepackage{amssymb}
\usepackage{enumerate}
\usepackage{tikz}
\usepackage{tikz-cd}
\usepackage[font=small,labelfont=bf]{caption}
\usepackage{hyperref} % Required for hyperlinks
\usepackage{cleveref}

\newtheorem{theorem}{Theorem}[section]
\newtheorem{conjecture}[theorem]{Conjecture}
\newtheorem{proposition}[theorem]{Proposition}
\newtheorem{lemma}[theorem]{Lemma}

\theoremstyle{remark}
\newtheorem{remark}{Remark}[theorem]
\numberwithin{equation}{section}

\newcommand{\Z}{\mathbb{Z}}
\newcommand{\Q}{\mathbb{Q}}

\newcommand{\OO}{\mathcal{O}}
\newcommand{\C}{\mathbb{C}}

\newcommand{\p}{\mathfrak{p}}
\newcommand{\q}{\mathfrak{q}}
\newcommand{\f}{\mathfrak{f}}
\newcommand{\oo}{\mathfrak{o}}
\newcommand{\fP}{\mathfrak{P}}
\newcommand{\PP}{\mathbb{P}}
\newcommand{\A}{\mathbb{A}}

\newcommand{\Gal}{\mathrm{Gal}}
\newcommand{\Frob}{\mathrm{Frob}}
\newcommand{\Disc}{\mathrm{Disc}}
\newcommand{\Spec}{\mathrm{Spec}}
\newcommand{\Li}{\mathrm{Li}}
\newcommand{\1}{\textbf{1}}

\renewcommand{\epsilon}{\varepsilon} % Better epsilon
\renewcommand{\leq}{\leqslant}
\renewcommand{\le}{\leqslant}
\renewcommand{\geq}{\geqslant}
\renewcommand{\ge}{\geqslant}

\title{Extreme values of Artin $L$-functions at $s=1$}

\author{Asif Zaman}
\author{Haonan Zhao}

\address{Department of Mathematics, University of Toronto,  Toronto ON, Canada}
\email{asif.zaman@utoronto.ca}
\email{zhaohaonan1228@gmail.com}

\begin{document}

\begin{abstract}
    Let $K/\Q$ be a Galois extension of number fields with Galois group $G$. Let $\chi$ be any character on $G$ and $\kappa(\chi)$ be the leading coefficient of the Laurent series of the Artin $L$-function $L(s,\chi)$ at $s=1$. Cho, Lemke Oliver, and Zaman (2026) recently established a conditional bound for $\kappa(\chi)$ assuming GRH for the Dedekind zeta function of $K$. Up to implied constants, we show that the extremal values of these bounds are attained for an infinite family of extensions $K/\Q$ and all irreducible faithful characters $\chi$, provided $G$ is a regular Galois group over $\Q$. All characters $\chi$ are permitted without restriction when  $G$ is abelian, $A_n, S_n$, or any sporadic group. Our method removes the assumptions of GRH and Artin's holomorphy conjecture from an argument of Duke (2003) by following the strategy of Lemke Oliver, Thorner, and Zaman (2024) on the existence of extremal class numbers for number fields of arbitrary degree. Key ingredients include progress by D\'{e}bes (2017) towards Malle's conjecture, and an Artin $L$-function large sieve by Lemke Oliver and Smith (2024).  
\end{abstract}

\maketitle

\section{Introduction}

Let $K/k$ be a Galois extension of number fields with Galois group $G$. Let $\rho:G \to \mathrm{GL}(V)$ be any finite-dimensional complex representation with character $\chi$, and denote $L(s,\chi)$ to be the corresponding Artin $L$-function. By the classical work of Artin and Hecke, it is known that the function $(s-1)^{\langle \chi,\1_G\rangle}L(s,\chi)$ is holomorphic and non-vanishing at $s=1$, where $\langle \cdot,\cdot \rangle$ is the inner product on $G$ and $\1_G$ is the trivial character on $G$. Define the quantity 
$$
\kappa(\chi) := \lim_{s \to 1} (s-1)^{\langle \chi,\1_G\rangle}L(s,\chi)
$$
to be the leading (non-zero) coefficient of the Laurent series of  $L(s,\chi)$ at $s=1$. When $\chi$ is trivial, this coefficient specializes to the residue of the Dedekind zeta function of $k$, made famous by Dirichlet's class number formula. This residue has been extensively studied with results on its extremal size, upper and lower bounds, and distribution in certain families, e.g., \cite{Duke03,Daileda-2006,Cho-2014, ChoKim-2017,LOTZ24,ChoKim-2026}.  

When $\chi$ is arbitrary, the quantity $\kappa(\chi)$ has been studied to a lesser extent.  Cho--Lemke Oliver--Zaman \cite{cho_effective_2025}   recently gave both unconditional and conditional bounds for  $\kappa(\chi)$ in terms of the discriminant $D_k = |\mathrm{Disc}(k)|$ of $k$, the analytic conductor $q(\chi)$, the values of $\chi$, and the values of the character $\widetilde{\chi}$ induced from the pullback of $\chi$ to $\operatorname{Gal}(\widetilde K/\mathbb{Q})$, where $\widetilde K$ is the Galois closure of $K$ over $\mathbb{Q}$. Assuming the generalized Riemann hypothesis (GRH) for the Dedekind zeta function $\zeta_K$,  they showed in  \cite[Proposition 2.1]{cho_effective_2025} that  
\begin{equation}
\label{eqn:uppbound}
    |\kappa(\chi)| \ll_{[k:\Q],|G|,\chi(1)} (\log \log 3 D_k)^{\widetilde{\chi}(1)-\chi(1)} (\log \log 3q(\chi))^{\chi(1)-\langle \chi,1\rangle}
\end{equation}
and
\begin{equation}
\label{eqn:lowbound}
    |\kappa(\chi)| \gg_{[k:\Q],|G|,\chi(1)} (\log \log 3 D_k)^{{\mu}(\widetilde{\chi})-\mu(\chi)} (\log \log 3q(\chi))^{\mu(\chi)-\langle \chi,1\rangle},
\end{equation}
where $\langle \, \cdot , \cdot \, \rangle $ is the normalized inner product on $G$, and 
\begin{equation}\label{def:mu}
\mu(\chi) = \min_{g \in G} \Re (\chi(g)) \quad \text{and} \quad \chi(1) =  \max_{g \in G} \Re (\chi(g)).
\end{equation}
By assuming Artin's holomorphy conjecture (AHC) in addition to GRH, the same bounds follow from earlier work of Duke \cite[Proposition 5]{Duke03} when $k= \Q$. 

Our goal is to investigate whether the bounds \eqref{eqn:uppbound} and \eqref{eqn:lowbound} are tight in an infinite family of fields with the same Galois group over $\Q$. This leads to a natural conjecture.

\begin{conjecture} \label[conjecture]{conj:main}
    Fix a non-trivial finite group $G$ and a subset $\mathcal{S} \subseteq \mathrm{Char}(G)$ of its characters. For any $g \in G$, there exist infinitely many normal extensions  $K/\Q$ with $\Gal(K/\Q) \simeq G$ such that 
    \[
    |\kappa(\chi)|  \asymp_{|G|,\chi(1)} (\log \log 3q(\chi))^{\Re \chi(g)-\langle\chi,1 \rangle}
    \]
    for all characters $\chi \in \mathcal{S}$. 
\end{conjecture}
By \eqref{def:mu}, a verification of this conjecture implies \eqref{eqn:uppbound} and \eqref{eqn:lowbound} are tight for the given finite group $G$ and subset $\mathcal{S} \subseteq \mathrm{Char}(G)$ of its characters. The choice $G = \Z/2\Z$ and $\mathcal{S} = \{\chi_\text{sgn}\}$ corresponds to quadratic extensions over $\Q$ and their Dirichlet $L$-functions. This special case of Conjecture \ref{conj:main} was verified by Montgomery--Weinberger \cite{MW77}, matching the classical GRH bounds of Littlewood \cite{Lit28}, and substantial distributional information has been extracted by Granville--Soundararajan \cite{granville_distribution_2003} and Lamzouri \cite{Lamzouri-2010,Lamzouri-2015}. 

For general groups $G$ and subsets $\mathcal{S} \subseteq \mathrm{Char}(G)$, this conjecture has significant barriers.  First and foremost, it requires progress towards the inverse Galois problem. This can be solved when $G$ is a regular Galois group over $\Q$, i.e.  $G$ is isomorphic to the Galois group of a regular extension $E$ over the rational function field $\Q(t)$ as defined by Serre \cite[Page 35]{serre_topics_2016}.  By assuming AHC and GRH, an argument of Duke \cite[Propositions 3 and 5] {Duke03}  can be adapted to conditionally verify Conjecture \ref{conj:main} when $G$ is a regular Galois group over $\Q$.  This broad class of groups is widely studied (see, e.g., \cite{serre_topics_2016}, \cite{MalleMatzat}, \cite{wittenberg25}, \cite{Huang26}).   

    The assumptions of AHC and GRH avoid deep technical barriers: Artin $L$-functions are not yet known to be holomorphic in general, and the coefficient $\kappa(\chi)$ may not be well-behaved unless the zeros of the Artin $L$-function are far from $s=1$. Our main theorem removes these assumptions provided the characters are irreducible and faithful. 

\begin{theorem}
    \label{thm:main1} Conjecture \ref{conj:main} holds if $G$ is a regular Galois group over $\mathbb{Q}$, and $\mathcal{S} \subset \mathrm{Char}(G)$ is the subset of faithful irreducible characters of $G$. 
\end{theorem}

We broadly follow the strategy of Lemke Oliver--Thorner--Zaman \cite{LOTZ24} for the existence of large class numbers for number fields of any degree. First, we invoke a quantitative inverse Galois theorem of D\'{e}bes \cite{De17} to construct a large family of number fields with the same Galois group $G$ by specializing a regular $G$-cover over $\PP^1_\Q$. With the Serre specialization theorem, we are further able to prescribe Frobenius elements at small primes. Second, we replace GRH and AHC with new large sieve results due to Lemke Oliver--Smith \cite{oliver_faithful_2024} for Artin $L$-functions (in some cases like $G=A_n$ or $S_n$, the zero density result in \cite{LOTZ24} would suffice).  This input ensures that 100\% of the  Artin $L$-functions attached to faithful irreducible representations in our family can be approximated by a short Euler product.

By iteratively applying the above method to all covers corresponding to all normal subgroups, our second main theorem allows all characters (not necessarily irreducible or faithful) provided $G$ has few proper normal subgroups, such as $A_n, S_n,$ and all sporadic groups.

\begin{theorem}
    \label{thm:main2} Conjecture \ref{conj:main} holds if $G$ is a regular Galois group over $\Q$, $\mathcal{S} = \mathrm{Char}(G)$, and 
    \[
    \sum_{N \triangleleft G, N\ne G} |N|<|G|.
    \]
\end{theorem}

\noindent
The above theorem notably excludes many Abelian groups, which we separately treat. 
\begin{theorem}
    \label{thm:main3} Conjecture \ref{conj:main} holds if $G$ is Abelian and $\mathcal{S} = \mathrm{Char}(G)$. 
\end{theorem}

 The method of Theorem \ref{thm:main2} fails because the theorem of D\'{e}bes (\Cref{thm:debes})  cannot construct enough fields for this special case (c.f. \Cref{rk: boundNotTight}). To prove Theorem \ref{thm:main3}, we explicitly build our family of Abelian extensions through a product of cyclic extensions. For these cyclic extensions of order $m$, we use prime cyclotomic fields $\Q(\zeta_q)$ to construct them. To prescribe the Frobenius elements at a small prime $p$ in these extensions, it turns out that a sufficient condition is to prescribe the power residue symbol $\big( \frac{p}{\q}\big)_m$ where $\q$ is a prime ideal lying over $q$ in $\Q(\zeta_m)$. With no simple reciprocity theorem as in the quadratic case,  it does not appear sufficient to select a suitable arithmetic progression from Sun Tzu's remainder theorem (c.f. \Cref{rk: quadraticCRT}); instead, we carefully apply the weak reciprocity with Kummer extensions and a uniform unconditional version of the Chebotarev density theorem due to Thorner--Zaman \cite{TZ19}. 

The proofs of our three theorems establish Conjecture \ref{conj:main} in a quantitative form; that is, we exhibit a family of fields $\mathfrak{F}(X)$ with absolute discriminant $\leq X$ such that $|\mathfrak{F}(X)|$ grows as a small power of $X$ as $X \to \infty$.

\subsection*{Organization} 

\Cref{sec:HilbertIrreducibility} summarizes the main analytic and algebraic geometric ingredients we use in the proof. Using these inputs, \Cref{sec:ProofOfSec2} proves two intermediate theorems towards \Cref{thm:main1,thm:main2}. \Cref{sec:AbelianCase} develops the corresponding results when $G$ is a finite Abelian group. \Cref{sec:ProofOfMainThms} finishes the proof of \Cref{conj:main} in the cases we considered. \Cref{sec:Example} provides a concrete construction when $G=S_n$ for $n\ge 5$.

\section{Hilbert irreducibility, prescribing Frobenius and short sum estimate}
\label{sec:HilbertIrreducibility}

 The key ingredient for \Cref{thm:main1,thm:main2} is the construction of a family of Artin $L$-functions through Hilbert irreducibility whereby $\kappa(\chi)$ can be approximated by a short Euler product (or a short sum after taking the logarithm). Further, to control this Euler product, we would like to control Frobenius elements so that they intersect the desired conjugacy classes.  In this section, we establish the following two main theorems.

\begin{theorem}
    \label{thm:shortEulerFaith}
    Let $G$ be a finite regular Galois group over $\Q$. For any conjugacy classes $[c_p]$ in $G$, there are infinitely many number fields $K$ with isomorphism $\Gal(K/\Q) \xrightarrow[]{\sim}G$. For any faithful, irreducible character $\chi\ne \1_G$ on $G$, we have
    $$
    \log L(1,\chi) = \sum_{p\le (\log q(\chi))^{\delta}}  \chi(p) p^{-1} +O_{G,\chi}(1)
    $$
    for some $\delta>0$, where $\chi(p) = 1/|I_p| \sum_{j \in I_p} \chi(\Frob_p j)$. Further, there is a constant $\lambda=\lambda(G)$ such that for all $\lambda \le p \le (\log q(\chi))^{\delta}$, we have $p$ unramified and $[\Frob_p] = [c_p]$.
\end{theorem}

If all characters are considered, by decomposing them into the irreducible components, we will prove the following.

\begin{theorem}
    \label{thm:shortEuler}
    If $G$ is further assumed to satisfy $\sum_{N \triangleleft G, N\ne G} |N|<|G|$, the infinite family can be chosen so that for any irreducible character $\chi\ne \1_G$ on $G$, we have
    $$
    \log L(1,\chi) = \sum_{p\le (\log q(\chi))^{\delta}}  \chi(p) p^{-1} +O_{G,\chi}(1)
    $$
    for some $\delta>0$, where $\chi(p) = 1/|I_p| \sum_{j \in I_p} \chi(\Frob_p j)$. Further, there is a constant $\lambda=\lambda(G)$ such that for all $\lambda \le p \le (\log q(\chi))^{\delta}$, we have $p$ unramified, $[\Frob_p] = [c_p]$, and $\log q(\chi) \asymp_G \log q(\psi)$ for any $\chi,\psi$ non-trivial irreducible characters.
\end{theorem}

We shall assume standard properties for Artin $L$-functions (see, e.g.,  Murty--Murty \cite{Murty-Murty-1997} for a concise introduction). Now, we outline the main inputs to obtain these results.

\subsection{Hilbert irreducibility}

First, we give a quick review of Hilbert irreducibility. Let $E/\Q(t)$ be a finite Galois extension with group $G$. $E$ is regular if $E \cap \overline{\Q} = \Q$. In this setting, let $V \to \PP^1_\Q$ be the corresponding $G$-cover, that is a finite, generically \'etale morphism with $\Gal(K(V)/K(\PP^1))=G$. In the geometric language, $E$ being regular is the same as $V$ being geometrically irreducible. Concretely, we may take $V$ to be the normalization of $\PP^1$ in $E$. For a variety $V$, we denote $V(k)$ to be the $k$-points, and we call a subset $A \subset V(k)$ thin if it is as defined in \cite[Definition 3.1.1]{serre_topics_2016}.  

\begin{theorem}[Hilbert Irreducibility]
    \label{thm:hilbert}
    There is a thin set $A \subset \PP^1(\Q)$ such that for all $t_0 \notin A$, the following diagram is Cartesian,
\[\begin{tikzcd}
	{\mathrm{Spec}E_{t_0}} && V \\
	{\mathrm{Spec} \Q} && \PP^1
	\arrow[from=1-1, to=1-3]
	\arrow[from=1-1, to=2-1]
	\arrow["\ulcorner"{anchor=center, pos=0.125}, draw=none, from=1-1, to=2-3]
	\arrow[from=1-3, to=2-3]
	\arrow["t_0"', from=2-1, to=2-3]
\end{tikzcd}\]
 with $E_{t_0}/\Q$ a finite Galois extension and $$\Gal(E_{t_0}/\Q)=G.$$
\end{theorem}

\begin{proof}
    See \cite[Proposition 3.3.1]{serre_topics_2016}.
\end{proof}

Let $\mathrm{Int}_A(T)$ be the number of integral points in the thin set $A$ with absolute value less than $T$, relative to the standard model $\A^1_\Z$. We record a quantitative result of Cohen \cite{Coh81}.

\begin{theorem}[S.D. Cohen]
\label{thm:cohen}
    $\mathrm{Int}_A(T) \ll_A T^{1/2} \log T$.
\end{theorem}

Additionally, the following result due to D\'ebes shows that many of the fields we obtain this way are in fact distinct.

\begin{theorem}[D\'ebes]
    \label{thm:debes}
    In the above setting, let $\mathcal{H}(T)$ be a collection of integers $t_0 \le T$ that satisfy Hilbert irreducibility and $H(t) = |\mathcal{H}(T)|$. There are $\gamma, \beta >0$, that only depend on $E$, such that the number of distinct fields we get, denoted by $N(T)$, is bounded below as
    $$
    N(T) \ge \frac{H(T)-\beta}{T^{1/|G|} (\log  T)^\gamma} .
    $$
\end{theorem}

\begin{proof}
    See \cite[Theorem 1.3]{De17}.
\end{proof}

\begin{remark}
\label{remark:poly}
    We discussed everything in terms of a Galois cover over $\PP^1$. In many cases, polynomials $f(x,t) \in \Z[t,x]$ are considered, and we set $E$ to be the splitting field of $f$ over $\Q(t)$. Let $W = \Spec\big( \Q[x,t]/(f) \big)$ and $V = W^\Gal$ be the Galois closure; then this reduces to the above situation by taking the completion of the curves. See \cite[Proposition 3.3.5]{serre_topics_2016} for more details.
\end{remark}

\subsection{Prescribing Frobenius}
We can control local Frobenius classes for sufficiently large $p$ using Serre specialization. Using the Hasse–Weil bound for curves over finite fields, Serre showed that it can be controlled by imposing a condition on $t_0 \mod p$. 

\begin{proposition}[Serre]
\label[proposition]{prop:serre}
    In the above setting, let $p_i$ be a finite sequence of primes and $[c_{p_i}]$ be a sequence of conjugacy classes in $G$ indexed by $p_i$. Then there is a constant $\lambda$, which depends on the cover (in fact, an integral model of the cover $V$), such that for all $p_i \ge \lambda$ we know $p_i$ is unramified and
    $$
    [\Frob_{p_i}] = [c_{p_i}]
    $$
    provided that $t_0 \equiv t_{p_i} \mod p_i$ for all $p_i$. Here, $\Frob_p$ denotes the corresponding Frobenius element for the extension $E_{t_0}/\Q$.
\end{proposition}

\begin{proof}
    See \cite[Section 4.6]{serre_topics_2016} or proof of \cite[Proposition 3]{Duke03}.
\end{proof}

\subsection{Estimate via short sum}

As in Montgomery--Weinberger \cite{MW77} or Lemke Oliver--Thorner--Zaman \cite{LOTZ24}, the estimate of $\log L(1,\chi)$ by a short sum can be obtained on average without GRH. A recent large sieve result by Lemke Oliver--Smith \cite{oliver_faithful_2024} yields such an estimate for most faithful irreducible $L$-functions without assuming GRH or AHC.

\begin{proposition}[Lemke Oliver, Smith]
\label[proposition]{prop:density}
    Let $k$ be a number field, $G$ be a finite group and $A>0$. Let $Q \ge 1$ and $\epsilon > 0$. Then, apart from at most $O_{k,G,\epsilon,A}(Q^\epsilon)$ exceptional fields $K$, each faithful, irreducible Artin L-function $L(s,\chi)$ attached to $\Gal(K/k)$ for a Galois $G$-extension $K$ with $D_K \le Q$ satisfies
    $$
    \log L(1,\chi) = \sum_{N(\p)\le (\log D_K)^A} \log L_\p(1,\chi) +O_{k,G,\epsilon,A}(1)
    $$
    where $L_\p(1,\chi) = \det(1-N(\p)^{-1}\Frob_\p|V^{I_\p})^{-1}$ is the local Artin factor.
\end{proposition}

\begin{proof}
    See \cite[Lemma 7.6 and the proof of Corollary 1.16]{oliver_faithful_2024}.
\end{proof}

\section{Proof of Theorems 2.1 and 2.2} 
\label{sec:ProofOfSec2}
We begin with the proof of \Cref{thm:shortEulerFaith}, as the proof of \Cref{thm:shortEuler} will be an iterative application of the ideas involved in \Cref{thm:shortEulerFaith}. Both proofs can be viewed as a generalization of Theorem 10.5 in \cite{LOTZ24}.

\begin{proof}[Proof of \Cref{thm:shortEulerFaith}] Since $G$ is regular Galois over $\Q$, we consider the geometrically irreducible Galois cover $V \to \PP^1$ with group $G$. Fix conjugacy classes $[c_p]$ in $G$ as in the theorem. Applying \Cref{prop:serre} with $p_i \le  (\log T)^{1/2}$ and $[c_{p}]$, we know there is a constant $\lambda$ together with $t_p$ such that if $t_0 \equiv t_p \mod p$, and if we specialize to $E_{t_0}$,  $\Frob_p = [c_p]$ for all $\lambda \le p \le  (\log T))^{1/2}$ and $p$ is unramified.

As in \cite[Theorem 10.5]{LOTZ24}  we define a congruence class $a \mod M$ by
$$
M=\prod_{\lambda<p\le (\log T))^{1/2}} p ,\text{and } a \equiv t_p \mod p
$$
and we note that $M = \exp(O((\log T)^{1/2}))$ by the prime number theorem.

For any $T>0$, we consider all $t_0 \le T$ with $t_0 \equiv a \mod M$. Using \Cref{thm:cohen}, it follows that there are at most $O(T^{1/2} \log T)$ exceptions with a specialized Galois group that differs from $G$. Combining these implies that we still have $\gg_\epsilon T^{1-\epsilon}$ for all $t_0 \le T$. Now, taking $\mathcal{H}(T)$ to be the ${t_0}$ satisfies above, \Cref{thm:debes} shows that we produced at least $\gg_\epsilon T^{1-1/|G| -\epsilon}$ distinct number fields. Further, by taking an integral monic polynomial model, it is easy to see that $D_{E_{t_0}} \ll_G T^A$ by \cite{Mah64}, for example. Now using \Cref{prop:density}, for any faithful irreducible character $\chi$, we have 

$$
\log L(1,\chi) = \sum_{p\le \log D_{E_{t_0}}} \log L_p(1,\chi) +O_{G,\epsilon}(1)
$$
with at most $O_{G,\epsilon'}(T^{\epsilon'A})$ exceptional number fields. Taking $\epsilon'$ small enough with $A\epsilon'<1-1/|G|-\epsilon$, we still have   $\gg_\epsilon T^{1-1/|G| -2\epsilon}$ distinct number fields with the above property.

Using the conductor-discriminant formula, we know that $\log q(\chi) \le \log D_{E_{t_0}}$. Hence, we can choose any $\delta<\frac{1}{2}$ and enlarge $\lambda$ so that $(\log q(\chi))^\delta\le (\log T)^{1/2}$ for any $(\log q(\chi))^\delta \ge \lambda$. Thus, with $\log q(\chi)\asymp_G  \log D_{E_{t_0}}$ from \cite[Lemma 4.2]{cho_effective_2025}, Mertens' formula and the fact that $\log L_p(1,\chi)  = \frac{\chi(p)}{p}+O_G(p^{-2})$, it follows that
$$
\log L(1,\chi) = \sum_{p\le (\log q(\chi))^\delta} \frac{\chi(p)}{p} +O_{G,\epsilon}(1)
$$
for the remaining $E_{t_0}$, which is what we need.
\end{proof}

Although a simpler proof can in fact be obtained by directly using \cite[Theorem 1.1]{De17} instead of \Cref{thm:debes}, this approach is insufficient for the proof of \Cref{thm:shortEuler}.

Now, to prove \Cref{thm:shortEuler}, we need to consider all irreducible characters. Let $\chi$ be any non-trivial irreducible character and $\rho_\chi$ be a representation with character $\chi$ which is determined uniquely up to isomorphism. Let $N_\chi = \ker \rho_\chi$. Then, $\rho_\chi$ factorizes through the quotient $G/N_\chi$, and we let the factored representation be $\rho_\chi^\dagger$ with character $\chi^\dagger$. Now, $L(s,\chi)=L(s,\chi^\dagger)$ and $\chi^\dagger$ is irreducible and faithful. Hence, \Cref{prop:density} may be applied, but this crucially requires all subfields $E_{t_0}^N$ to change along $t_0$. An iterative application of \Cref{thm:debes} overcomes this issue of changing subfields.

\begin{proof}[Proof of \Cref{thm:shortEuler}]
    As above, we fix a geometrically irreducible Galois cover $V \to \PP^1$ with group $G$. For any proper normal subgroup $N \triangleleft G$, we consider its corresponding cover $V^N \to \PP^1$. We fix conjugacy classes $[c_p]$ in $G$ as above, and we know for any $T>0$, there are at least $\gg_\epsilon T^{1-\epsilon}$ such that $t_0 \le T$, $t_0$ satisfies Hilbert irreducibility for all covers $V^N$ and $\Frob_p = [c_p]$ for all $\lambda<p\le (\log T)^{1/2}$. Now, it follows that the diagram below is Cartesian for these $t_0$'s, with corresponding fields being $E_{t_0}^N$ for all $N \triangleleft G$.
    \[\begin{tikzcd}
	{\Spec E_{t_0}} && V \\
	{\Spec E_{t_0}^N} && {V^N} \\
	{\Spec \Q} && {\PP^1}
	\arrow[from=1-1, to=1-3]
	\arrow[from=1-1, to=2-1]
	\arrow["\ulcorner"{anchor=center, pos=0.125}, draw=none, from=1-1, to=2-3]
	\arrow[from=1-3, to=2-3]
	\arrow[from=2-1, to=2-3]
	\arrow[from=2-1, to=3-1]
	\arrow["\ulcorner"{anchor=center, pos=0.125}, draw=none, from=2-1, to=3-3]
	\arrow[from=2-3, to=3-3]
	\arrow[from=3-1, to=3-3]
    \end{tikzcd}\]

We apply \Cref{thm:debes} to the cover $V \to \PP^1$ with these $t_0$. For each distinct number field, choosing a representative $t_0 \in \Z$ implies that there are still $\gg_\epsilon T^{1-1/|G| -\epsilon}$ of them. Then, we may iterate this to the cover $V^N \to \PP^1$. Hence, after running over all proper $N \triangleleft G$, we remain to have $\gg_\epsilon T^{1-\sum_{N \triangleleft G, N\ne G} 1/[G:N] -\epsilon}$ many $t_0 \le T$ and it is clear that for these $t_0$, $E_{t_0}^N$ are all distinct for any $N \triangleleft G$. As before, we also have $D_{E_{t_0}} \ll_G T^A$. Since $\sum_{N \triangleleft G, N\ne G} |N| < |G|$, it follows that $\sum_{N \triangleleft G, N\ne G} 1/[G:N] <1$. By choosing $\epsilon$ sufficiently small, we know that $1-\sum_{N \triangleleft G, N\ne G} 1/[G:N] -3\epsilon>0$. Taking $\epsilon'A < 1-\sum_{N \triangleleft G, N\ne G} 1/[G:N] -\epsilon$, by \Cref{prop:density}, any faithful irreducible character $\chi$ on $G/N$, $\log L(1,\chi)$ can be approximated by a short sum 
$$
\log L(1,\chi) = \sum_{p\le \log D_{E_{t_0}^N}} \log L_p(1,\chi) +O_{G,\epsilon}(1),
$$
with at most $O_{G,\epsilon'}(T^{\epsilon'A})$ exceptions. Hence, by removing these for all $N$, we have
$$
\gg_\epsilon T^{1-\sum_{N \triangleleft G, N\ne G} 1/[G:N] -2\epsilon}
$$
such $t_0$ left, which clearly gives an infinite family as $T \to \infty$. 

Using the bound
$$
\#\big\{E_{t_0}^N: t_0<T, D_{E_{t_0}^N}<T^\eta\big\} \ll_G T^{|G|\eta},
$$
we may further choose $|G| \eta < 1-\sum_{N \triangleleft G, N\ne G} 1/[G:N] -2\epsilon$ so that after removing these fields, we still have 
$$
\gg_\epsilon T^{1-\sum_{N \triangleleft G, N\ne G} 1/[G:N] -3\epsilon}
$$
desired $t_0$ and $\log D_{E_{t_0}^N} \asymp_G \log T$.

We now claim these $t_0$ work as in the theorem. Indeed, take any irreducible character $\chi$ and let $N:=N_\chi$ be the kernel of its representation. Hence, $\chi$ descents to a  faithful irreducible $\chi^\dagger$ on $G/N$ and
$$
\log L(1,\chi) =\log L(1,\chi^\dagger)= \sum_{p\le \log D_{E_{t_0}^N}} \log L_p(1,\chi^\dagger) +O_{G,\epsilon}(1).
$$
Since now we also have 
$$\log q(\chi) =\log q(\chi^\dagger) \asymp_G \log D_{E_{t_0}^N} \asymp_G \log T,$$
we know that for $\delta < \frac{1}{2}$ and $\lambda$ sufficiently large, $(\log q(\chi))^\delta\le (\log T)^{1/2}$ for $(\log q(\chi))^\delta \ge \lambda$. Thus, we have for all $\lambda \le p \le (\log q(\chi))^{\delta}$,  $p$ is unramified and $[\Frob_p] = [c_p]$. Further, by $ \log q(\chi)\asymp_G  \log D_{E_{t_0}^{N}}$, Mertens' formula and $ \log L_p(1,\chi) = \frac{\chi(p)}{p} +O_G(p^{-2})$, it follows that
$$
\log L(1,\chi) = \sum_{p\le (\log q(\chi))^\delta} \frac{\chi(p)}{p} +O_{G,\epsilon}(1)
$$
for the remaining $E_{t_0}$. So, the theorem follows by observing that $\log q(\chi)\asymp_G \log T$.
\end{proof}

\section{The finite Abelian case}
\label{sec:AbelianCase}
Note that $G$ being regular Galois over $\Q$ and $\sum_{N \triangleleft G, N\ne G} |N| < |G|$ does not cover all Abelian cases. Indeed, the density argument is inherently easier for groups with fewer normal subgroups. Abelian groups, on the other hand, have many normal subgroups, which requires a more precise control of how all subfields vary. We consider a different family instead of directly using Hilbert irreducibility. This section is dedicated to establishing an analogue of \Cref{thm:shortEuler} when $G$ is finite Abelian.
\begin{theorem}
    \label{thm:shortEulerAb}
    Let $G$ be a finite Abelian group. For any $c_p$ in $G$, there are infinitely many number fields $K$ with isomorphism $\Gal(K/\Q) \xrightarrow[]{\sim}G$. For any irreducible character $\chi\ne 1$ on $G$, we have
    $$
    \log L(1,\chi) = \sum_{p\le (\log q(\chi))^{\delta}}  \chi(p) p^{-1} +O_{G,\chi}(1)
    $$
    for some $\delta>0$, where $\chi(p) = 1/|I_p| \sum_{j \in I_p} \chi(\Frob_p j)$. Further, there is a constant $\lambda=\lambda(G)$ such that for all $\lambda \le p \le (\log q(\chi))^{\delta}$, we have $p$ unramified, $\Frob_p = c_p$, and $\log q(\chi) \asymp_G \log q(\psi)$ for any $\chi,\psi$ non-trivial irreducible characters.
\end{theorem}

The cyclic case $G = \Z/m\Z$ is pivotal. Here we consider all prime $q \equiv 1 \mod m$ with $q\le T$ and we have $\Gal(\Q(\zeta_q)/\Q) \cong \Z/(q-1)\Z$. Hence, it has a unique subfield $E_q$ with $\Gal(E_q/\Q) \cong \Z/m\Z$ after we identify $\Gal(\Q(\zeta_q)/\Q) \cong \Z/(q-1)\Z$. For simplicity, we also fix roots of unity $\zeta_n$ once for all. We begin with controlling the Frobenius elements for these fields $E_q$.

\subsection{Prescribing Frobenius}
Throughout this subsection,  fix $G = \Z/ m \Z$. The main difficulty is to control how many such primes $q\le T$ satisfy $(\frac{E_q/\Q}{p}) = c_p \in G$ for all small primes $p$ with finitely many exceptions. Let $S(a)$ consist of all primes dividing $m$ and $a$.

\begin{lemma}
    \label[lemma]{lm:Reciprocity}
    Keep the notation as above and let $K = \Q(\zeta_m)$ and $F_p = K(p^{1/m})$ and let    
    $$C_p = \Big\{\sigma \in \Gal(F_p/K) : \sigma(p^{1/m}) = \zeta_m^{c_p} p^{1/m}\Big\}$$
    If a prime ideal $\q \in \Spec( O_K^{S(p)} )$ with $N(\q)=q$ being a rational prime and $(\frac{F_p/K}{\q}) \in C_p \subset \Gal(F_p/K)$, then $q \equiv 1 \mod m$ and there is an identification of $\Gal(\Q(\zeta_q)/\Q) \cong \Z/(q-1)\Z$ depends only on $\q$ such that $(\frac{E_q/\Q}{p}) = c_p$. 
\end{lemma}

\begin{proof}
    Since $N(\q)=q$ and $q$ is unramified, we know that $q$ splits completely. Hence, we have $q \equiv 1 \mod m$. Now, as $(\frac{F_p/K}{\q}) \in C_p$, using weak reciprocity, we have
    $$
    \big(\frac{F_p/K}{\q}\big)(p^{1/m}) = \left(\frac{p}{\q}\right)_m p^{1/m} = \zeta_m^{c_p} p^{1/m},
    $$
    where $\left( \frac{\cdot}{\cdot}\right)_m$ is the $m$-th power residue symbol. 
    Thus, we have $\left(\frac{p}{\q}\right)_m = \zeta_m^{c_p}$.
    
    Choosing a primitive root $\alpha \mod q$ such that 
    $$
    \alpha^\frac{q-1}{m} \equiv \zeta_m \mod \q,
    $$
    we identify $\sigma=(\zeta_q \mapsto \zeta_q^\alpha) \in \Gal(\Q(\zeta_q)/\Q)$ with  $1 \in \Z/(q-1)\Z$. By observing that $(\frac{\Q(\zeta_q)/\Q}{p}) = (\zeta_q \mapsto \zeta_q^p)$, we may write $p = \alpha^j \mod q$ and
    $$
    \zeta_m^{c_p} = \left(\frac{p}{\q}\right)_m = p^\frac{q-1}{m} = \alpha^\frac{j(q-1)}{m} = \zeta_m^j \mod \q.
    $$
    Hence, $j=c_p \mod m$ and the result follows as $(\frac{\Q(\zeta_q)/\Q}{p}) \equiv \sigma^{c_p} \mod \sigma^m$
\end{proof}

Thus, the above result turns the count of prescribed Frobenius at a fixed prime $p$ into a Chebotarev density estimate in the prime ideals $\q$ of $\Q(\zeta_m)$. Since multiple $p_i$'s are considered, we denote $F = F_{p_1} \cdots F_{p_r}$ to be the compositum. Now, as noted by Thorner-Zaman \cite[Theorem 1.4]{TZ19}, the relative quantity needed to control is
$$
Q(F) = \max_{\chi \text{ irred. in } \Gal(F/K)} N(\mathfrak{f}_\chi), 
$$
where $\f_\chi$ is the relative conductor of $\chi$. We can estimate $\mathfrak{f}_\chi$ for irreducible $\chi$ on $\Gal(F_p/K)$.

\begin{lemma}
    For any $p \nmid m$ and $\chi$ any irreducible character on $\Gal(F_p/K)$, we have $\mathfrak{f}_\chi \mid \mathfrak{C}_m \prod \p$ where the product is over all $\p \mid p$ and $\mathfrak{C}_m$ is supported on primes dividing $m$ and is uniform in $p$.
\end{lemma}
\begin{proof}
    Since $F_p = K(p^{1/m})$ and $|\Disc(x^m - p)| = m^m p^{m-1}$, we know that
    $$
    \oo_{F_p/K} \mid m^m p^{m-1},
    $$
    where $ \oo_{F_p/K}$ is the relative discriminant. Hence, we may take $\mathfrak{C}_m = m^m$ by the conductor-discriminant formula. On the other hand, the only other $\p$'s that are ramified are those $\p \mid p$. However, since $p \nmid m$, the ramification is tame. Thus,
    $
    \nu_\p(\f_\chi) \le 1
    $
    and the result follows.
\end{proof}

\begin{lemma}
    \label[lemma]{lm:conductorBd}
    Taking all $p_i \nmid m$, then $Q(F) \le C_m (\prod_{i} p_i)^{\varphi(m)}$ for an absolute $C_m$ that only depends on $m$.
\end{lemma}
\begin{proof}
    Indeed, since $F_{p_i}$'s are linearly disjoint, any irreducible character $\chi$ on $\Gal(F/K)$ is a tensor product $\chi_1 \otimes \cdots \otimes \chi_r$, where $\chi_i$ can be viewed on $\Gal(F_{p_i}/K)$ after projection. Using \cite[Lemma 6.6]{LOTZ24} inductively, we see that $\f_\chi \mid \mathrm{lcm}(\f_{\chi_1},...,\f_{\chi_r}) \mid \mathfrak{C}_m \prod_\p \p$ where the product is taking over all $\p \mid p_i$ for some $i$. The result then follows by taking the norm.
\end{proof}

We can give the density estimate from here.

\begin{proposition}
    \label[proposition]{prop:densityAb}
    Keep the same assumptions as above; let
    $$
    N(T) = \# \Big \{q \le T:\big(\frac{E_q/\Q}{p}\big) = c_p \text{, unramified for } p\le \sqrt{\log T} \text{, } p \nmid m\text{ and } q \equiv 1 \mod m \Big \} 
    $$
    Then, $N(T) \gg_{m,\lambda} T^\lambda$ for any $\lambda<1$.
\end{proposition}

\begin{proof}
    We use the same notation as above, and we take $p_i$ to be all primes less than $\sqrt{\log T}$. In view of \Cref{lm:Reciprocity}, it is enough to count primes $\q$ in $K$ with $m(\prod_i p_i)\le N(\q)\le T$, $\q$ has residue degree 1 and $(\frac{F/K}{\q}) = c_p$ where these $c_p \in C_p$ after projection. 
    Now, we observe the following facts:
    \begin{itemize}
        \item Using the prime number theorem, we have $\prod_i p_i = \exp (O (\sqrt{ \log T}))$
        \item  There are only $O(\sqrt T \log T)$ many $\q$ with $f(\q \mid q)>1$.
        \item There is at most a $\varphi(m)$-to-1 relationship between $\q$ and $(q) = \q \cap \Z$.
    \end{itemize}
    Hence, the result follows once we establish that the following quantity
    $$
    \Lambda(T) = \# \Big \{ N(\q) \le T: (\frac{F/K}{\q}) =c_p  \text{ for all } p\le \sqrt{\log T} \text{ with } p \nmid m \Big \}
    $$
    satisfies $\Lambda(T) \gg_{m,\lambda} T^\lambda$ for any $\lambda<1$.
    For this, we directly apply an absolute Chebotarev density theorem due to \cite[Theorem 1.4]{TZ19} with $x = T$ to the extension $F/K$. The constants $c_i$ are as in \cite{TZ19}. Since $Q(F) \le C_m (\prod p)^{\varphi(m)} \le C_m \exp(O_m(\sqrt{\log T})) \ll_{m,\epsilon} T^\epsilon$ by \Cref{lm:conductorBd}, we know that $T \ge (D_K Q(F) [K:\Q]^{[K:\Q]})^{c_2}$ for large $T$. Hence, the theorem applies and
    $$
    \Lambda(T) = \frac{\Li(T) -\theta_1 \Li(T^{\beta_1})}{[F:K]} \left( 1+ O\Bigg( \exp \Big[ \frac{-c_3 \log T}{ \log(D_K Q(F) n_K^{n_K})} \Big] + \exp\Big [ \frac{-(c_3 \log T)^{1/2}}{n_K^{1/2}}\Big ] \Bigg) \right),
    $$
    where $\beta_1$ is the potential Siegel zero, $\theta_1 \in \{-1,0,1\}$ and $n_K = [K:\Q]$.
    Using the bound for $Q(F)$, we see that
    $$
    \Lambda(T) \gg_m (\Li(T) - \theta_1 \Li(T^{\beta_1}))/[F:K].
    $$
    Since $[F:K] \le m^{\pi(\sqrt{\log T})} = T^{o(1)}$ by the prime number theorem, it is enough to show that $(\Li(T) - \theta_1 \Li(T^{\beta_1}) \gg_{m,\lambda} T^{\lambda}$ for any $\lambda<1$. Now, the worse case scenario is when $\theta_1 = 1$ and $\beta_1$ exists and in this case, we have
    $$
    \Li(T) -\Li(T^{\beta_1}) \gg \frac{1}{\log T} (T - T^{\beta_1}) = \frac{T}{\log T} (1-T^{\beta_1-1}) = \frac{T}{\log T} (1-\exp((\beta_1-1)\log T) ) 
    $$
    using the mean value theorem. Hence,
    $$
    \Li(T) -\Li(T^{\beta_1}) \gg \frac{T}{\log T} \min(1,(1- \beta_1) \log T) = \min(\frac{T}{\log T},T (1-\beta_1)).
    $$
    Now, using Stark's bound (e.g. \cite[Theorem 3.3]{TZ19}), we know that 
    $$
    1-\beta_1 \gg_m \frac{1}{Q(F)^{2} \log Q(F)}.
    $$
    The result follows by recalling that $Q(F) \ll_\epsilon T^\epsilon $ for any $\epsilon$.
\end{proof}

\begin{remark}
    \label[remark]{rk: quadraticCRT}

    We compare the above results with the classical Dirichlet $L$-functions. Let $G = \Z/2\Z$. In this case, the above construction yields $E_q = \Q(\sqrt{q*})$ where $q$ is a prime and 
    $$
    q* = (-1)^{\frac{q-1}{2}} q.
    $$
    Hence, if we let $\chi$ be the sign character, then $L(s,\chi) = L(s,\chi_{q_*})$ where $\chi_{q^*}(a) = \left( \frac{q*}{a} \right)_K$ is given by the Kronecker symbol. The local Frobenius condition at $p$ is precisely the same as whether $q^*$ is a residue or non-residue mod $p$ by the quadratic reciprocity. Using the Chinese Remainder Theorem, that's the same as some congruence classes mod $4 \prod p_i$. Hence, the density estimate can be given by the number of primes in an arithmetic progression. 
\end{remark}

\subsection{The finite Abelian case}

Now, we are able to finish the proof of \Cref{thm:shortEulerAb}.

\begin{proof}[Proof of \Cref{thm:shortEulerAb}]
    Since $G$ is finite Abelian, by the fundamental theorem of finite Abelian groups, we know that $G \cong \Z/m_1\Z \times \cdots \times \Z/m_r \Z$. Now, take $T>0$ and apply \Cref{prop:densityAb} to all $\Z/m_i \Z$. From this, we can obtain $E_{q_{i}^j}$ with $\Gal(E_{q_i^j}/\Q) \cong \Z/m_i \Z$ and $E_{q_i^j}$ only ramifies at $q_i^j$. Further, it also shows that if we consider the $q_i^j \le T$,  $(\frac{E_{q_i^j}/\Q}{p})=c_p$ and  unramified for $p \nmid m$, $p \le \sqrt{\log T}$,   we have $\gg_{G,\lambda} T^\lambda$ for each $i$ and $\lambda<1$. For $\eta |G|<1$, using the bound 
    $$
    \#\big\{E_{q_i^j}: {q_i^j} \le T, D_{E_{q_i^j}}<T^\eta\big\} \ll_G T^{|G|\eta},
    $$
    we may further restrict to these $D_{E_{q_i^j}} \ge T^\eta$ and we still have $\gg_{G,\lambda} T^\lambda$ many for each $i$ and $\lambda<1$.

    Now, for any $k \ge 1$, we can choose ${\textbf q}^k = (q_1^{u_1},...,q_r^{u_r})$ with $i k\le u_i\le(i+1) k$ and $q_i^{u_i}\ne q_j^{u_j}$ for all $i \ne j$. Clearly, there are still $\gg_{G,\lambda} T^\lambda$ many ${\textbf q}^k$ with $m_i<{\textbf q}^k_i \le T$ for all $1 \le i\le r$. We consider $E_k = E_{{\textbf q}^k_1}\cdots E_{{\textbf q}^k_r}$. Since $E_{{\textbf q}_i^k}$ ramifies at different primes, these $E_{{\textbf q}_i^k}$ are linearly disjoint which implies that $\Gal(E_k/\Q) \cong G$. 
    
    Moreover, we may bound the discriminant of $E_{{\textbf q}^k_i}$ from above. Since $E_{{\textbf q}^k_i} \subset \Q(\zeta_{{\textbf q}^k_i})$ and $\Q(\zeta_{{\textbf q}^k_i})$ only ramifies at ${{\textbf q}^k_i}$, $\nu_p(D_{E_{{\textbf q}^k_i}}) = 0$ for all $p\ne {{\textbf q}^k_i}$. But as the ${{\textbf q}^k_i}>m_i$, the ramification at ${{\textbf q}^k_i}$ is tame. Hence, we know that $D_{E_{{\textbf q}^k_i}} \mid ({{\textbf q}^k_i})^{m_i}$ and $D_{E_{{\textbf q}^k_i}} \le T^{m_i}$. Using the bound for the compositum, it follows that $D_{E_k} \le T^{A_G}$ for a fixed constant $A_G$. Hence, $\log D_{E_k} \asymp_G \log T$.
    
    Now, the result follows by the same averaging argument once we show that every $N <G$, $E_k^N$ varies. Indeed, under the identification of $\Gal(E_k/\Q)$ and $G \cong \Z/m_1\Z \times \cdots \times \Z/m_r \Z$, we know that $$I_{{\textbf q}_i^k} = \{0\} \times \Z/m_i \Z \times \{0\},$$ where $I_{{\textbf q}_i^k}$ denotes the inertia group at ${\textbf q}_i^k$. Hence, we know that for any $N< G$, if the image of $ \{0\} \times \Z/m_i \Z \times \{0\} \to G/N$ is nontrivial, $E_k^N$ is ramified at ${\textbf q}_i^k$. But since we have chosen ${\textbf q}^k$ such that all entries are increasing in $k$, we know from the ramification data that for any fixed $N < G$, $E_k^N$ are different for different $k$.

    Applying \Cref{prop:density} to all $E_k^N$, we know that for each $N$, there are only $O_{G,\epsilon}(T^{\epsilon A_G})$ exceptions and for the rest $k$'s, we have
    $$
    \log L(1,\chi) = \sum_{N(\p)\le (\log D_{E_k^N})} \log L_p(1,\chi) +O_{G,\epsilon}(1)
    $$
    for all non-trivial irreducible characters on $G/N$. Once applying this to all $N <G$, since we have $\gg_{\lambda} T^{\lambda}$ of such $k$'s and we only remove $O_{G,\epsilon}(T^{\epsilon A_G})$ many $k$'s for each $N<G$, there are still $\gg_{\lambda} T^{\lambda}$ many such $k$'s left if $\epsilon$ is chosen to be sufficiently small. Now, as in the earlier proofs,  
    \[
    \log q(\chi) \asymp_G \log D_{E_k^{N_\chi}} \asymp_G \log T
    \]
    for any irreducible $\chi$ on $G$. So, there is a $\delta$ such that $(\log q(\chi))^\delta \le \sqrt{\log T}$ for large $T$. The result follows as before by descending irreducible characters $\chi$ on $G$ to faithful irreducible characters on $G/N_\chi$, noticing $ \log q(\chi) \asymp_G \log  D_{E_k^N} \asymp_G \log T$, $\log L_p(1,\chi) = \chi(p)/p +O(p^{-2})$, and applying Mertens' formula. 
\end{proof}

\section{Proof of Theorems 1.2, 1.3, and 1.4}
\label{sec:ProofOfMainThms}
Now, we can easily finish the proof of \Cref{thm:main1}, \Cref{thm:main2}, and \Cref{thm:main3} from  \Cref{thm:shortEulerFaith}, \Cref{thm:shortEuler}, and \Cref{thm:shortEulerAb}, respectively. Since the proofs are essentially the same, we only illustrate the idea through the proof of \Cref{thm:main2}.

\begin{proof}[Proof of \Cref{thm:main2}]
    First, assume $\chi$ is a non-trivial irreducible character. Taking $[c_p] = [g]$ for all $p$, by \Cref{thm:shortEuler}, we know for such $\chi$, we have
    $$
    \log |L(1,\chi)| = \Re \sum_{p\le (\log q(\chi))^{\delta}}  \chi(p) p^{-1} +O_{G,\chi}(1)
    $$
    for some $\delta>0$. Further, there is a constant $\lambda=\lambda(G)$ such that for all $\lambda \le p \le (\log q(\chi))^{\delta}$, $p$ is unramified and we have $[\Frob_p] = [g]$.

    Hence, we know that
\begin{align*}
    \log |L(1,\chi)| &= \Re \sum_{\lambda \le p\le (\log q(\chi))^{\delta}}  \chi(\Frob_p) p^{-1} +O_{G,\chi}(1) \\
    &=\Re \chi(g) \sum_{\lambda \le p\le (\log q(\chi))^{\delta}} \frac 1 p +O_{G,\chi}(1) \\
    &=\Re \chi(g) \log \log \log  q(\chi) +O_{G,\chi}(1) 
\end{align*}
    by Mertens' formula. By exponentiation, we have
    \[
    |L(1,\chi)|  \asymp_{|G|,\chi(1)} (\log \log 3 q(\chi))^{\Re \chi(g)}.
    \]

    For general non-trivial character $\chi$, we write $$\chi = c_0 \1_G+\sum_{\psi \ne \1_G \text{ irred.}} c_\psi \psi.$$
    We may assume that one of the $c_\psi \ne 0$, as the theorem trivially follows otherwise. Now, we know that the above construction supplies an infinite family of fields such that
    $$
    |L(1,\psi)|  \asymp_{|G|,\psi(1)} (\log \log 3 q(\psi))^{\Re \psi(g)}
    $$
    for all nontrivial irreducible characters $\psi$.
    Hence,
    $$
    |\kappa(\chi)| = \Big|\prod_{\psi \ne \1_G \text{ irred.}} L(1,\psi)^{c_\psi}\Big | \asymp_{|G|,\chi(1)} \prod_{\psi \ne \1_G \text{ irred.}} (\log \log 3 q(\psi))^{c_\psi \Re \psi(g)}
    $$
    Now, as $\log q(\psi) \asymp_G \log T$  by \Cref{thm:shortEuler}, we know that 
    $$\log q(\chi) = \sum_{\psi \ne \1_G \text{ irred.}} c_\psi \log q(\psi) \asymp_{G,\chi(1)} \log T.$$
    Hence, we have
    $$
    |\kappa(\chi)|\asymp_{|G|,\chi(1)} (\log \log 3 q(\chi))^{\sum_{\psi \ne \1_G \text{ irred.}} c_\psi \Re \psi(g)} = (\log \log 3 q(\chi))^{\Re \chi(g)-\langle\chi,1 \rangle}.
    $$
\end{proof}

\section{One concrete example: $S_n$ for $n \geq 5$}
\label{sec:Example}
Lastly, we demonstrate how the Hilbert irreducibility method can be more concretely used in a particular cover that is generated by a polynomial, in this case $S_n$ for $n \geq 5$ for simplicity. This serves a few purposes. First, it shows a way to generalize \Cref{thm:main2} to other fields that are not covered. Second, we will show that the number of fields we obtain through a particular polynomial is much more than the ones given in \Cref{thm:main2} through the proof of \Cref{thm:shortEuler}. Third, we will briefly discuss the change of base field as brought up in \cite{cho_effective_2025}. 

The exact polynomial of this depends on the parity of $n$. To begin with, we start with $n$ odd. The main difficulty is to make sure all Galois subfields vary with $t_0$. To make this easy, the following polynomial is considered (see \Cref{remark:poly}).
$$
f(x,t) = x^n - x^{n-1} - t
$$

First, we need to show the following.

\begin{lemma}
    \label[lemma]{lm:regular}
    The splitting field generated by $f(x,t)$ over $\Q(t)$ is regular Galois with group $S_n$. Consequently, the corresponding cover $V \to \PP^1$ is a $G$-cover and geometrically irreducible. 
\end{lemma}

\begin{proof}
    This is the example on page 42 in \cite{serre_topics_2016}, which has ramification type $(n,n-1,2)$.
\end{proof}

We let $F_{t_0}=\Q(\alpha_0)$  where $\alpha_0$ is a root of $f_{t_0}(x)=f(x,t_0)$ over $\Q$. Then $E_{t_0}$ is the Galois closure of $F_{t_0}$. Next, we will study the unique Galois subfield that is fixed by the unique normal subgroup in $A_n$.

\begin{lemma}
\label[lemma]{lm:quadSub}
    If $t_0$ satisfies Hilbert irreducibility, i.e. $\Gal(E_{t_0}/\Q) =S_n$, then $$E_{t_0}^{A_n} = \Q(\sqrt{\pm D_{F_{t_0}}}).$$
\end{lemma}

\begin{proof}
    Let $\alpha_i$ be the roots of $f_{t_0}$ in $\overline{\Q}$. Clearly, we know that $\sqrt{\Disc(f)} = \prod_{i<j}(\alpha_i-\alpha_j)$ is stable under $A_n$ action but not under $S_n$ action. Hence, it follows that $E_{t_0}^{A_n} = \Q(\sqrt{\Disc(f)})$. Now, since we have
\begin{equation}
\label{eqn:discComp}
    \pm D_{F_{t_0}} [\OO_{F_{t_0}}:\Z[\alpha_0]]^2 =  \Disc(f),
\end{equation}
    the result follows.
\end{proof}

 Then we need to give precise control of the discriminant $D_{F_{t_0}}$.

\begin{lemma}
\label[lemma]{lm:disc2}
    If $t_0$ is a square free integer, then  for all $p\nmid n-1$ and $p \mid t_0$, we have $p^{n-2} \parallel D_{F_{t_0}}$.
\end{lemma}

\begin{proof}
    A direct application of Cohen \cite[Theorem 6.1.4.b ]{Cohen1993} shows that the order $\Z[\alpha_0]$ is $p$-maximal for any $p \mid t_0$. Hence, we are in a position to apply the Dedekind-Kummer theorem (e.g. \cite[Lemma 22]{Swinnerton-Dyer_2001}) to read off the ramification data of $p$ from the reduction of $f \mod p$. Since we know that
    $$
    f(x,t_0) = x^{n-1}(x-1) \mod p,
    $$
    we have $p \OO_{F_{t_0}} = \fP_1^{n-1} \fP_2$.  Now, since $p\nmid n-1$, $p$ is tame. Hence,
    $$
    p^{n-2} = p^{\sum f_i(e_i-1)} \parallel D_{F_{t_0}}
    $$
    by applying Theorem 21 in Swinnerton-Dyer \cite{Swinnerton-Dyer_2001}.
\end{proof}

Now, we can prove \Cref{thm:shortEuler} for $G=S_n$, $n$ odd with the particular cover generated by $f(x,t) = x^n - x^{n-1} - t$.

\begin{proof}[Proof of \Cref{thm:shortEuler} when $G=S_n$ for odd $n \geq 5$]
    We fix conjugacy classes $[c_p]$ in $G$ as in the theorem. The first part of the proof is similar to the general case. We apply \Cref{prop:serre} with $p_i \le  (\log T)^{1/2}$ and $[c_{p_i}]=[c_p]$. Then, we know there is a constant $\lambda=\max (\lambda(V\to \PP^1),n+1)$ together with $t_p$ such that if $t_0 \equiv t_p \mod p$, and if we specialize to $E_{t_0}$, we have $[\Frob_p] = [c_p]$ for all $\lambda \le p \le  (\log T))^{1/2}$ and $p$ is unramified. We  also define the similar congruence class $a \mod M$ by
    $$
    M=\prod_{p\le (\log T))^{1/2}} p ,\text{and } a \equiv t_p \mod p \text{ for } p> \lambda \text{ and } a \equiv 1 \mod p \text{ for } p\le \lambda
    $$
    and $M = \exp(O((\log T)^{1/2}))$ by the prime number theorem.

    But now, instead of using \Cref{thm:debes}, we can utilize the property of the polynomial to separate the distinct fields generated through specialization of the cover. For any $T$ sufficiently large, we consider all $t_0 \le T$ square-free, with $t_0 \equiv a \mod M$. Using \Cref{thm:cohen}, there are at most $O(T^{1/2} \log T)$ exceptions with a specialized Galois group that differs from $S_n$. Further, we have $t_0 \mid D_F$ by \Cref{lm:disc2}. Using Malher's bound for the discriminant(e.g. \cite{Mah64}) and \Cref{eqn:discComp}, it follows that $D_{F_{t_0}} \ll_n t_0^{2n(n-1)} \ll_n T^{2n(n-1)}$. Since $d(N)\ll_\epsilon N^\epsilon$ (e.g. \cite[2.20]{Montgomery_Vaughan_2006}), we know that there is a $1$-to-$O_n(T^\epsilon)$ relation between $t_0$ and $D_{F_{t_0}}$. Moreover, for each $K/\Q$ extension with $\Gal(K/\Q)=S_n$, there are at most $2^{n!}$ subfields. Hence, the number of distinct $t_0$ square-free that give distinct $S_n$-fields $E_{t_0}$ is bounded below by $\gg_{n,\epsilon} N(T) T^{-\epsilon}$, where $N(T)$ is the number of square-free integers $t_0$ less than $T$ that are congruent to $a(T) \mod M(T)$ and $p\nmid t_0$ for all $p \mid n$. Now, using the standard estimate for $N(T)$, we know that $N(T) \gg_\epsilon T^{1-\epsilon}$. Hence,  there are at least $\gg_{n,\epsilon} T^{1-\epsilon}$ distinct square-free $t_0$ that give us distinct distinct $S_n$-fields $E_{t_0}$. 

    As in the proof of \Cref{thm:shortEuler}, the proof now consists of two steps to finish. In \textit{Step 1.}, we apply the density result \Cref{prop:density} to faithful irreducible characters, and in \textit{Step 2.}, we iterate the application of \Cref{prop:density} to the quotient of $S_n$ by $A_n$ to deal with the sign character.

    \textit{Step 1.} For these fields $E_{t_0}$, we know that $D_{t_0}\ll_n D_{F_{t_0}}^{n!}$, where we use $D_{t_0} = D_{E_{t_0}}$ to simplify notation. Hence, we have $D_{t_0}\ll_n t_0^{2n(n-1)n!} \ll_n T^{2n(n-1)n!}$. Apply \Cref{prop:density} to the extension $E_{t_0}/\Q$ with any $\epsilon<1/2n(n-1)n!$, and $\chi$ any non-trivial irreducible character different from the sign character; it follows that
    $$
    \log L(1,\chi) =  \sum_{p\le \log D_{t_0}} \log L_p(1,\chi) +O_{G,\epsilon}(1)
    $$
    for all $t_0$ as above with $O_{|G|,\epsilon}(T^{\epsilon 2n(n-1)n!})$ exceptions.
    Hence, we still have $\gg_{n,\epsilon} T^{1-\epsilon}$ many $t_0$ left.

    For these representations,  they are faithful. Hence, \cite[Lemma 4.2]{cho_effective_2025} and the conductor-discriminant formula show that $\log q(\chi) \asymp_{n} \log D_{t_0}$. Hence, we may choose any $\delta<\frac{1}{2}$ and enlarge $\lambda$ so that $(\log q(\chi))^\delta\le (\log T)^{1/2}$ for any $(\log q(\chi))^\delta \ge \lambda$. 

    Now, using $\log L_p(1,\chi) = \chi(p) p^{-1}+O(p^{-2})$, we see that
    $$
    \log L(1,\chi) =\sum_{p\le \log D_{t_0}} \chi(p) p^{-1} + O_{G,\chi,\epsilon}(1)
    $$
     Hence, by Mertens' formula and $\log q(\chi) \asymp_n \log D_{t_0}$ again, we obtain the desired estimates for any character different from $\chi_\text{sgn}$.

     \textit{Step 2.} Now, let's deal with the case of $\chi_\text{sgn}$ by further thinning the set of $t_0$'s. Let $\rho_\text{sgn}$ denote the sign representation. We know that $\rho_\text{sgn}:S_n \to \C^\times$ factorizes through the quotient $S_n/A_n = \Z/2\Z$ and we know that $L(1,\chi_\text{sgn}) = L(1,\chi)$ where $\chi$ denotes the unique non-trivial representation of the quadratic subfield $L:=E_{t_0}^{A_n} = \Q(\sqrt{\pm D_{F_{t_0}}})$ by \Cref{lm:quadSub}. Hence, we iterate the application of \Cref{prop:density} to the extension $L/\Q$ as before. To do so, we must first show that this field varies with $t_0$ again.

     We begin by estimating the size of $D_L$. By \Cref{lm:disc2}, \Cref{lm:quadSub}, the fact that $n$ is odd and $p \nmid t_0$ for all $p < n$, we know that for all $p \mid t_0$, $p$ is ramified in $L$. Hence, we know $t_0 \mid D_L$. Using the relative discriminant formula, we know that $D_L \ll_n t_0^{2n(n-1)n!} \ll_n T^{2n(n-1)n!}$. Thus, $\log D_L \asymp_n \log D_{t_0}$. Since $\chi_\text{sgn}$ is faithful in the quotient, we once again know that $\log q(\chi) \asymp_n \log D_L \asymp_n \log D_{t_0}$. Hence, by increasing $\lambda$ if needed, we can make all $\chi$ irreducible characters such that $(\log q(\chi))^\delta \le (\log T)^{1/2}$ for  any $(\log q(\chi))^\delta \ge \lambda$.   
     
     Now, arguing as above, we still have the  $1$-to-$O_n(T^\epsilon)$ relation between $t_0$ and $D_{L}$. By a similar application of \Cref{prop:density} as above, 
    \begin{align*}
        \log L(1,\chi) &= \sum_{p\le \log D_L} \log L_p(1,\chi) + O_{G,\chi,\epsilon}(1) \\
        &=\sum_{p\le \log D_L} \chi(p) p^{-1} + O_{G,\chi,\epsilon}(1)
    \end{align*}
    for $\gg_{n,\epsilon} T^{1-\epsilon}$ number of $t_0$. 
    Hence, the estimate follows by Mertens' formula and $\log q(\chi) \asymp_n \log D_{L}$. The result follows once we choose $\eta n! <1$ and remove all $t_0$ such that $D_{t_0} \le T^\eta$ and $D_L \le T^\eta$ to ensure $\log q(\chi) \asymp_n \log T$ for all non-trivial irreducible characters $\chi$.
\end{proof}

The case for even $n \geq 6$ can be proven in a similar way. We just need to use the following polynomial instead.
$$
f(x,t) = (x-t)(x-2^2t)(...)(x-n^2t)-t
$$
The regular Galois result is shown in \cite[Proposition 2]{Duke03}, replacing \Cref{lm:regular}. For square-free $t_0$, we observe that $f_{t_0}$ is $p$-Eisenstein for all $p \mid t_0$ and therefore $p$ are totally ramified. This observation replaces \Cref{lm:disc2}.

\begin{remark}
    The application of \Cref{prop:density} in this proof can be replaced by the zero-density result of Lemke Oliver, Thorner and Zaman. Indeed, one can use \cite[Lemma 3.2, Theorem 3.1, and Proposition 10.3]{LOTZ24} in place of \Cref{prop:density}.
\end{remark}

\begin{remark}
    \label[remark]{rk: boundNotTight}
    We note that the bound given by \Cref{thm:debes} is not sharp at all in this example. By controlling the ramification of the total field and the subfield carefully, this method produces $\gg_\epsilon T^{1-\epsilon}$ many fields through specialization, whereas the general method only produces $\gg_\epsilon T^{1/2-1/n!-\epsilon}$ many fields. 
\end{remark}

\begin{remark}
    For other $G$ that is regular Galois over $\Q$, if one can obtain a good control on the number of distinct fields generated (likely through ramification similar to above), the condition $$\sum_{N \triangleleft G, N \ne G} |N| <|G|$$ becomes superfluous.
\end{remark}

\begin{remark}
    We would also like to address the change of base field as brought up in \cite[Section 2.3]{cho_effective_2025}. Taking the same example as in their paper, we consider an $S_3$ field $K/\Q$ with unique quadratic field $F$ and let $\psi$ be any non-trivial character in $\Gal(K/F)$ and $\chi$ be the character for the standard representation in $S_3$. If we denote $G=\Gal(K/\Q)$ and $N = \Gal(K/F)$, as observed by \cite{cho_effective_2025}, we have $\chi = \mathrm{Ind}_{N}^{G} \psi$. Hence, we have
    $$
    \kappa(\chi) = L(1,\chi) = L(1,\psi) = \kappa(\psi).
    $$
    Thus, we obtain two different GRH bounds. Using the fact that $\chi, \psi$ are faithful, we really have
    \begin{equation}
        \label{eq: boundOverQ}
        \frac{1}{\log \log D_K} \ll L(1,\chi) \ll (\log \log D_K)^2
    \end{equation}
    when the base field is $\Q$ and
    \begin{equation}
        \label{eq: boundOverL}
        \frac{1}{(\log \log D_K)^{1/2}(\log \log D_F)^{1/2}} \ll L(1,\psi) \ll (\log \log D_K) (\log \log D_F)
    \end{equation}
    if the base field is considered to be $F$. Although the bound (\ref{eq: boundOverL}) seems to be strictly stronger than (\ref{eq: boundOverQ}), this improvement disappears for families attaining the true extremal value. In fact, if we attain the extreme value by forcing all Frobenius elements in $K$ to land in the same class, then the discriminant of $F$ is necessarily large.  For such families, under GRH, we always have $\log \log D_K \asymp \log \log D_F$, and this is true for any normal tower $K/F/\Q$ by the GRH Chebotarev density theorem. Hence, the extreme values are the same for induced characters when changing the base field.
\end{remark}

\begin{remark}
    \label[remark]{rk:attainLower}
    One natural question to ask is whether there is a family that attains (\ref{eq: boundOverL}) but not (\ref{eq: boundOverQ}). The answer is positive. For this, we will fix $F=\Q(\zeta_3)$ and use the Kummer extension $K_t=F(\sqrt[3]{t})$ for cubic free $t$'s. Clearly, we have $\Gal(K_t/\Q)=S_3$ and $K_t^{A_3} = F$ because $K_t$ is the Galois closure of the field $\Q(\sqrt[3]{t})$. For these $K_t$'s, we can consider them as specializations of the Galois cover over $\PP^1_F$ generated by $x^3-t$. Note that this extension is regular, since both its geometric and arithmetic monodromy groups are isomorphic to $\Z/3\Z$. Hence, we can apply Serre's specialization (c.f. \Cref{prop:serre}) to this extension for one of the prime ideals $\p \subset \OO_{\Q(\zeta_3)}$ lying over $p$ where $p$ splits in $\Q(\zeta_3)/\Q$. It follows that the upper and lower bounds of (\ref{eq: boundOverL}) are attained using the zero density estimate in \cite{LOTZ24}, Serre's specialization, and cubic-free integer density. On the other hand, because $F$ is fixed, we know that the bounds in (\ref{eq: boundOverQ}) are strictly weaker.
\end{remark}

\section*{Acknowledgements}
\noindent
This project was organized and partly funded by the Department of Mathematics at the University of Toronto as part of the 2026 Summer Undergraduate Student Research Awards.  The first author thanks Peter J. Cho and Robert Lemke Oliver for their encouragement. The second author thanks Ruizhen Liu for proofreading part of an earlier version of the manuscript.  AZ was partially supported by NSERC grant RGPIN-2022-04982. 

\section*{AI Declaration}
\noindent
No text in the manuscript is written by AI. ChatGPT was used to search the literature and proofread the manuscript. GPT-6 Astra produced the example in \Cref{rk:attainLower}.

\bibliographystyle{alpha}
\bibliography{references}

@book {Montgomery_Vaughan_2006, 
    place={Cambridge}, 
    series={Cambridge Studies in Advanced Mathematics}, 
    title={Multiplicative Number Theory I: Classical Theory}, publisher={Cambridge University Press}, 
    author={H. L. Montgomery and R. C. Vaughan}, 
    year={2006}, 
    collection={Cambridge Studies in Advanced Mathematics}
}

@Inbook {Cohen1993,
    author="Cohen, Henri",
    title="Algorithms for Algebraic Number Theory II",
    bookTitle="A Course in Computational Algebraic Number Theory",
    year="1993",
    publisher="Springer Berlin Heidelberg",
    address="Berlin, Heidelberg",
    pages="297--359",
    isbn="978-3-662-02945-9",
    doi="10.1007/978-3-662-02945-9_6",
    url="https://doi.org/10.1007/978-3-662-02945-9_6"
}

@book {Swinnerton-Dyer_2001, 
    place={Cambridge}, 
    series={London Mathematical Society Student Texts}, 
    title={A Brief Guide to Algebraic Number Theory}, 
    publisher={Cambridge University Press}, 
    author={Swinnerton-Dyer, H. P. F.}, 
    year={2001}, 
    collection={London Mathematical Society Student Texts}}

@book {serre_topics_2016,
	title = {Topics in Galois Theory},
	isbn = {9780429064401},
	url = {https://www.taylorfrancis.com/books/9781439865255},
	publisher = {A K Peters/{CRC} Press},
	author = {Serre, Jean-Pierre},
	urldate = {2026-07-04},
    year = {2016},
	date = {2016-04-19},
	langid = {english},
	doi = {10.1201/b10588},
}

@article {Duke03,
    AUTHOR = {Duke, W.},
     TITLE = {Extreme values of {A}rtin {$L$}-functions and class numbers},
   JOURNAL = {Compositio Math.},
  FJOURNAL = {Compositio Mathematica},
    VOLUME = {136},
      YEAR = {2003},
    NUMBER = {1},
     PAGES = {103--115},
      ISSN = {0010-437X,1570-5846},
   MRCLASS = {11R42 (11R29)},
  MRNUMBER = {1966783},
MRREVIEWER = {V.\ Kumar\ Murty},
       DOI = {10.1023/A:1022695505997},
       URL = {https://doi.org/10.1023/A:1022695505997},
}

@article {Daileda-2006,
    AUTHOR = {Daileda, Ryan C.},
     TITLE = {Non-abelian number fields with very large class numbers},
   JOURNAL = {Acta Arith.},
  FJOURNAL = {Acta Arithmetica},
    VOLUME = {125},
      YEAR = {2006},
    NUMBER = {3},
     PAGES = {215--255},
      ISSN = {0065-1036,1730-6264},
   MRCLASS = {11R29 (11F12 11M26 11R47)},
  MRNUMBER = {2276192},
MRREVIEWER = {Siman\ Wong},
       DOI = {10.4064/aa125-3-2},
       URL = {https://doi-org.myaccess.library.utoronto.ca/10.4064/aa125-3-2},
}

@article {LOTZ24,
    AUTHOR = {Lemke Oliver, Robert J. and Thorner, Jesse and Zaman, Asif},
     TITLE = {An approximate form of {A}rtin's holomorphy conjecture and
              non-vanishing of {A}rtin {$L$}-functions},
   JOURNAL = {Invent. Math.},
  FJOURNAL = {Inventiones Mathematicae},
    VOLUME = {235},
      YEAR = {2024},
    NUMBER = {3},
     PAGES = {893--971},
      ISSN = {0020-9910,1432-1297},
   MRCLASS = {11R42 (11R45)},
  MRNUMBER = {4701881},
MRREVIEWER = {Christopher\ Frei},
       DOI = {10.1007/s00222-023-01232-2},
       URL = {https://doi.org/10.1007/s00222-023-01232-2},
}

@article {MW77,
    AUTHOR = {Montgomery, Hugh L. and Weinberger, Peter J.},
     TITLE = {Real quadratic fields with large class number},
   JOURNAL = {Math. Ann.},
  FJOURNAL = {Mathematische Annalen},
    VOLUME = {225},
      YEAR = {1977},
    NUMBER = {2},
     PAGES = {173--176},
      ISSN = {0025-5831,1432-1807},
   MRCLASS = {12A50 (10H10)},
  MRNUMBER = {427271},
MRREVIEWER = {Matti\ Jutila},
       DOI = {10.1007/BF01351721},
       URL = {https://doi.org/10.1007/BF01351721},
}

@misc {cho_effective_2025,
      title="{E}ffective {Brauer-Siegel} theorems for {A}rtin {$L$-functions}",
      author={Cho, Peter J.  and Lemke Oliver, Robert J.  and Zaman, Asif },
      year={2026},
      eprint={2510.02309},
      archivePrefix={arXiv},
      primaryClass={math.NT},
      url={https://arxiv.org/abs/2510.02309}, 
      note={preprint, arXiv:2510.02309},
}

@article {Coh81,
    AUTHOR = {Cohen, S. D.},
     TITLE = {The distribution of {G}alois groups and {H}ilbert's
              irreducibility theorem},
   JOURNAL = {Proc. London Math. Soc. (3)},
  FJOURNAL = {Proceedings of the London Mathematical Society. Third Series},
    VOLUME = {43},
      YEAR = {1981},
    NUMBER = {2},
     PAGES = {227--250},
      ISSN = {0024-6115,1460-244X},
   MRCLASS = {12A20 (12A75)},
  MRNUMBER = {628276},
MRREVIEWER = {D.\ J.\ Lewis},
       DOI = {10.1112/plms/s3-43.2.227},
       URL = {https://doi.org/10.1112/plms/s3-43.2.227},
}

@article {Mah64,
    AUTHOR = {Mahler, K.},
     TITLE = {An inequality for the discriminant of a polynomial},
   JOURNAL = {Michigan Math. J.},
  FJOURNAL = {Michigan Mathematical Journal},
    VOLUME = {11},
      YEAR = {1964},
     PAGES = {257--262},
      ISSN = {0026-2285,1945-2365},
   MRCLASS = {12.30 (10.79)},
  MRNUMBER = {166188},
MRREVIEWER = {L.\ Mirsky},
       URL = {http://projecteuclid.org/euclid.mmj/1028999140},
}

@article {Lit28,
    AUTHOR = {Littlewood, J. E.},
     TITLE = {On the {C}lass-{N}umber of the {C}orpus {$P({\surd}-k)$}},
   JOURNAL = {Proc. London Math. Soc. (2)},
  FJOURNAL = {Proceedings of the London Mathematical Society. Second Series},
    VOLUME = {27},
      YEAR = {1928},
    NUMBER = {5},
     PAGES = {358--372},
      ISSN = {0024-6115},
   MRCLASS = {99-04},
  MRNUMBER = {1575396},
       DOI = {10.1112/plms/s2-27.1.358},
       URL = {https://doi.org/10.1112/plms/s2-27.1.358},
}

@misc {oliver_faithful_2024,
      title={{F}aithful {A}rtin induction and the {C}hebotarev density theorem}, 
      author={Lemke Oliver, Robert J.  and Smith, Alexander},
      year={2024},
      eprint={2405.08383},
      archivePrefix={arXiv},
      primaryClass={math.NT},
      url={https://arxiv.org/abs/2405.08383}, 
      note={preprint, arXiv:2405.08383},
}

@article {De17,
    AUTHOR = {D{\'e}bes, Pierre},
     TITLE = {On the {M}alle conjecture and the self-twisted cover},
   JOURNAL = {Israel J. Math.},
  FJOURNAL = {Israel Journal of Mathematics},
    VOLUME = {218},
      YEAR = {2017},
    NUMBER = {1},
     PAGES = {101--131},
      ISSN = {0021-2172,1565-8511},
   MRCLASS = {12F10 (12F12)},
  MRNUMBER = {3625127},
MRREVIEWER = {Chr.\ U.\ Jensen},
       DOI = {10.1007/s11856-017-1461-7},
       URL = {https://doi.org/10.1007/s11856-017-1461-7},
}

@article {granville_distribution_2003,
	 AUTHOR = {Granville, A. and Soundararajan, K.},
     TITLE = {The distribution of values of {$L(1,\chi_d)$}},
   JOURNAL = {Geom. Funct. Anal.},
  FJOURNAL = {Geometric and Functional Analysis},
    VOLUME = {13},
      YEAR = {2003},
    NUMBER = {5},
     PAGES = {992--1028},
      ISSN = {1016-443X,1420-8970},
   MRCLASS = {11M20},
  MRNUMBER = {2024414},
MRREVIEWER = {Philippe\ G.\ Michel},
       DOI = {10.1007/s00039-003-0438-3},
       URL = {https://doi.org/10.1007/s00039-003-0438-3},
}

@article {Lamzouri-2015,
    AUTHOR = {Lamzouri, Youness},
     TITLE = {Extreme values of class numbers of real quadratic fields},
   JOURNAL = {Int. Math. Res. Not. IMRN},
  FJOURNAL = {International Mathematics Research Notices. IMRN},
      YEAR = {2015},
    NUMBER = {22},
     PAGES = {11847--11860},
      ISSN = {1073-7928,1687-0247},
   MRCLASS = {11R29 (11R11)},
  MRNUMBER = {3456705},
MRREVIEWER = {Miho\ Aoki},
       DOI = {10.1093/imrn/rnv054},
       URL = {https://doi-org.myaccess.library.utoronto.ca/10.1093/imrn/rnv054},
}

@article {Lamzouri-2010,
    AUTHOR = {Lamzouri, Youness},
     TITLE = {Distribution of values of {$L$}-functions at the edge of the
              critical strip},
   JOURNAL = {Proc. Lond. Math. Soc. (3)},
  FJOURNAL = {Proceedings of the London Mathematical Society. Third Series},
    VOLUME = {100},
      YEAR = {2010},
    NUMBER = {3},
     PAGES = {835--863},
      ISSN = {0024-6115,1460-244X},
   MRCLASS = {11M06 (11N37)},
  MRNUMBER = {2640292},
MRREVIEWER = {Antanas\ Laurin\v cikas},
       DOI = {10.1112/plms/pdp050},
       URL = {https://doi-org.myaccess.library.utoronto.ca/10.1112/plms/pdp050},
}

@article {TZ19,
    AUTHOR = {Thorner, Jesse and Zaman, Asif},
     TITLE = {A unified and improved {C}hebotarev density theorem},
   JOURNAL = {Algebra Number Theory},
  FJOURNAL = {Algebra \& Number Theory},
    VOLUME = {13},
      YEAR = {2019},
    NUMBER = {5},
     PAGES = {1039--1068},
      ISSN = {1937-0652,1944-7833},
   MRCLASS = {11R44},
  MRNUMBER = {3981313},
MRREVIEWER = {Mario\ Pineda-Ruelas},
       DOI = {10.2140/ant.2019.13.1039},
       URL = {https://doi.org/10.2140/ant.2019.13.1039},
}

@article {ChoKim-2017,
    AUTHOR = {Cho, Peter J. and Kim, Henry H.},
     TITLE = {Extreme residues of {D}edekind zeta functions},
   JOURNAL = {Math. Proc. Cambridge Philos. Soc.},
  FJOURNAL = {Mathematical Proceedings of the Cambridge Philosophical
              Society},
    VOLUME = {163},
      YEAR = {2017},
    NUMBER = {2},
     PAGES = {369--380},
      ISSN = {0305-0041,1469-8064},
   MRCLASS = {11R42 (11M41)},
  MRNUMBER = {3682635},
MRREVIEWER = {Ayyadurai\ Sankaranarayanan},
       DOI = {10.1017/S0305004117000019},
       URL = {https://doi-org.myaccess.library.utoronto.ca/10.1017/S0305004117000019},
}

@article {ChoKim-2026,
    AUTHOR = {Cho, Peter J. and Kim, Gyeongseok},
     TITLE = {On the residues and {E}uler-{K}ronecker constants of cyclic
              number fields},
   JOURNAL = {Ramanujan J.},
  FJOURNAL = {Ramanujan Journal. An International Journal Devoted to the
              Areas of Mathematics Influenced by Ramanujan},
    VOLUME = {69},
      YEAR = {2026},
    NUMBER = {3},
     PAGES = {Paper No. 58, 32},
      ISSN = {1382-4090,1572-9303},
   MRCLASS = {11R42 (11M06)},
  MRNUMBER = {5031651},
       DOI = {10.1007/s11139-026-01331-7},
       URL = {https://doi-org.myaccess.library.utoronto.ca/10.1007/s11139-026-01331-7},
}

@article {Cho-2014,
    AUTHOR = {Cho, Peter Jaehyun},
     TITLE = {The strong {A}rtin conjecture and large class numbers},
   JOURNAL = {Q. J. Math.},
  FJOURNAL = {The Quarterly Journal of Mathematics},
    VOLUME = {65},
      YEAR = {2014},
    NUMBER = {1},
     PAGES = {101--111},
      ISSN = {0033-5606,1464-3847},
   MRCLASS = {11R29 (11M26 11R80)},
  MRNUMBER = {3179652},
MRREVIEWER = {Timothy\ All},
       DOI = {10.1093/qmath/has046},
       URL = {https://doi-org.myaccess.library.utoronto.ca/10.1093/qmath/has046},
}

@misc{wittenberg25,
      title={{P}ark {C}ity lecture notes: around the inverse {G}alois problem}, 
      author={Olivier Wittenberg},
      year={2025},
      eprint={2302.13719},
      archivePrefix={arXiv},
      primaryClass={math.NT},
      url={https://arxiv.org/abs/2302.13719}, 
      note={preprint, arXiv: 2302.13719}
}

@book {MalleMatzat,
    AUTHOR = {Malle, Gunter and Matzat, B. Heinrich},
     TITLE = {Inverse {G}alois theory},
    SERIES = {Springer Monographs in Mathematics},
 PUBLISHER = {Springer-Verlag, Berlin},
      YEAR = {1999},
     PAGES = {xvi+436},
      ISBN = {3-540-62890-8},
   MRCLASS = {12F12 (11R32 12F10 14G32 20C33 20F36 20G40)},
  MRNUMBER = {1711577},
MRREVIEWER = {Helmut\ V\"olklein},
       DOI = {10.1007/978-3-662-12123-8},
       URL = {https://doi.org/10.1007/978-3-662-12123-8},
}

@misc{Huang26,
      title={The {M}athieu group {$M_{23}$} is a {G}alois group over $\mathbb{Q}$}, 
      author={Xiaoyu Huang and Blake Jackson and Kyu-Hwan Lee and Bjorn Poonen and Rachel Pries and Shaowu Zhang},
      year={2026},
      eprint={2608.08538},
      archivePrefix={arXiv},
      primaryClass={math.NT},
      url={https://arxiv.org/abs/2608.08538},
      note={preprint, arXiv:2608.08538}
}

@book {Murty-Murty-1997,
    AUTHOR = {Murty, M. Ram and Murty, V. Kumar},
     TITLE = {Non-vanishing of {$L$}-functions and applications},
    SERIES = {Progress in Mathematics},
    VOLUME = {157},
 PUBLISHER = {Birkh\"auser Verlag, Basel},
      YEAR = {1997},
     PAGES = {xii+196},
      ISBN = {3-7643-5801-7},
   MRCLASS = {11Mxx (11F67 11G40 11R42)},
  MRNUMBER = {1482805},
MRREVIEWER = {D.\ R.\ Heath-Brown},
       DOI = {10.1007/978-3-0348-0274-1},
       URL = {https://doi-org.myaccess.library.utoronto.ca/10.1007/978-3-0348-0274-1},
}
\parindent0pt

\end{document}